\documentclass[a4paper,reqno,11pt]{amsart}

\usepackage[utf8]{inputenc} %
\usepackage{textcomp}
\usepackage[T1]{fontenc}

\usepackage[colorlinks,backref]{hyperref}
\usepackage{srcltx}
\usepackage{multicol}

\usepackage{fullpage}
\usepackage{setspace}
\usepackage{datetime}
\usepackage{url}

\usepackage{amsfonts} %
\usepackage{stmaryrd} %
\usepackage{amsmath, amssymb, amsthm}
\usepackage{thmtools}
\usepackage{bm}
\usepackage{bbm}
\usepackage{mathtools}
\usepackage{enumitem}
\usepackage{./atalhos}
\usepackage{./notation}
\usepackage[portuguese,ruled,noline,noend,linesnumbered]{algorithm2e}



\makeatletter
\def\@setdate{\datename\ \@date}
\def\@setaddresses{\par
  \nobreak \begingroup
\footnotesize
  \def\author##1{\nobreak\addvspace\bigskipamount}%
  \def\\{\unskip, \ignorespaces}%
  \interlinepenalty\@M
  \def\address##1##2{\begingroup
    \par\addvspace\bigskipamount\indent
    \@ifnotempty{##1}{(\ignorespaces##1\unskip) }%
    {\scshape\ignorespaces##2}\par\endgroup}%
  \def\curraddr##1##2{\begingroup
    \@ifnotempty{##2}{\nobreak\indent{\itshape Current address}%
      \@ifnotempty{##1}{, \ignorespaces##1\unskip}\/:\space
      ##2\par}\endgroup}%
  \def\email##1##2{\begingroup
    \@ifnotempty{##2}{\nobreak\indent{\itshape Email
    addresses}%
      \@ifnotempty{##1}{, \ignorespaces##1\unskip}\/:\space
      \ttfamily##2\par}\endgroup}%
  \def\urladdr##1##2{\begingroup
    \@ifnotempty{##2}{\nobreak\indent{\itshape URL}%
      \@ifnotempty{##1}{, \ignorespaces##1\unskip}\/:\space
      \ttfamily##2\par}\endgroup}%
  \addresses
  \endgroup
}
\makeatother

\def\({\left(}
\def\){\right)}  
\def\<{\langle}
\def\>{\rangle}
\let\:=\colon

\usepackage{datetime}
\usepackage{lineno}
\newcommand*\patchAmsMathEnvironmentForLineno[1]{%
\expandafter\let\csname old#1\expandafter\endcsname\csname #1\endcsname
\expandafter\let\csname oldend#1\expandafter\endcsname\csname end#1\endcsname
\renewenvironment{#1}%
{\linenomath\csname old#1\endcsname}%
{\csname oldend#1\endcsname\endlinenomath}}%
\newcommand*\patchBothAmsMathEnvironmentsForLineno[1]{%
\patchAmsMathEnvironmentForLineno{#1}%
\patchAmsMathEnvironmentForLineno{#1*}}%
\AtBeginDocument{%
\patchBothAmsMathEnvironmentsForLineno{equation}%
\patchBothAmsMathEnvironmentsForLineno{align}%
\patchBothAmsMathEnvironmentsForLineno{flalign}%
\patchBothAmsMathEnvironmentsForLineno{alignat}%
\patchBothAmsMathEnvironmentsForLineno{gather}%
\patchBothAmsMathEnvironmentsForLineno{multline}%
}

\begin{document}
\shortdate
\yyyymmdddate
\settimeformat{ampmtime}
\date{\today, \currenttime}
\onehalfspacing
\title{%
    Ramsey-type problems for orientations of graphs
}

\author[Bruno Pasqualotto Cavalar]{%
  Bruno Pasqualotto Cavalar*
}
\thanks{*The author acknowledges support from São Paulo Research
Foundation (FAPESP), grants \#2015/26678-9 and \#2018/05557-7.}
\address{Instituto de Matem\'atica e Estat\'{\i}stica, Universidade de S\~ao
  Paulo, Rua do Mat\~ao 1010, 05508--090~S\~ao Paulo, SP}
\email{brunopc@ime.usp.br}

\begin{abstract}
Given an acyclic oriented graph $\dgh$
and a graph $G$,
we write
$G \rms \dgh$
if
every orientation of $G$
has an oriented copy of
$\dgh$.
We define
$\dr(\dgh)$
as the smallest number $n$ such that there exists a graph
$G$ of order~$n$
satisfying $G \rms \dgh$.
Denoting by $R(H)$
the classical Ramsey number of a graph $H$,
we show that
$\dr(\dgh) \leq 2R(H)^{c \log^2 h}$
for every acyclic oriented graph $\dgh$ with $h$ vertices,
where $H$ is its underlying undirected graph.
We also study the threshold function for
the event 
$\{G(n,p) \rms \dgh\}$
in the binomial random graph.
Finally, we 
consider
the
{isometric} 
model,
and
we prove an upper bound for the
isometric Ramsey number of an acyclic orientation of the
cycle,
by applying the hypergraph container lemma in random
graphs
and
adapting an argument of
{H{\`a}n},
{Retter},
{R{\"o}dl}
and
{Schacht}.
\end{abstract}

\maketitle
\pagestyle{plain}
\footskip=25pt

\vspace{-20pt}
\section{Introduction}
Given graphs $G$ and $H$,
we write
$G \rms H$
to denote that
every
two-coloring of
the edges of $G$
contains a monochromatic copy of $H$. 
The Ramsey number $R(H)$ of a graph $H$ is 
defined as
\begin{equation*}
    R(H)
    :=
    \inf\set{
        n \in \bbn : 
        \text{there exists a graph } 
        G = G^n
        \text { such that }
        G \rms H
    },
\end{equation*}
where $G = G^n$ denotes that $G$ is a graph on $n$ vertices.
This number was proved to be finite 
by
Ramsey~\cite{ramsey_original}
and
Erd\H{o}s and Szekeres~\cite{ramsey_erdos2}.
Finding bounds
for $R(H)$ 
is a classical problem in combinatorics
(see e.g. the dynamic survey of 
Radziszowski~\cite{radz_survey_ramsey}).
Moreover, the threshold function for
the property that a random graph $G(n,p)$
satisfies
$G(n,p) \to H$
is well-studied~\cite{rodl_threshold}.
We
study the same problems
for 
a variant of this notion
for orientations of graphs.
Let us begin with a few definitions.

\subsection{Digraphs and oriented graphs}
A
\defx{directed graph}
or
\defx{digraph}
$\dgg$
is a pair
$\dgg = (V,E)$
where
$V$ 
is 
a set of vertices
and
$E$
is a set such that
$E \sseq (V \times V) \sm \set{(v,v) : v \in V}$.
Just as in the case
of undirected graphs,
an element of $E$ is called
an
\defx{edge};
however, it may also be called
an
\defx{arc}
to differ from the undirected case.
An 
\defx{oriented graph}
$\dgg=(V,E)$
is a digraph
where
$(u,v) \in E$
implies
$(v,u) \notin E$
for every 
$u,v \in V$.
Moreover,
an oriented graph
$\dgg = (V, E)$
is said to be
an
\defx{orientation}
of a graph
$G = (V',E')$
if
$V=V'$
and,
for every 
$u,v \in V=V'$,
we have
$\set{u,v} \in E'$
if and only if
$(u,v) \in E$
or
$(v,u) \in E$.
In this case,
we say that 
$G$
is the
\defx{underlying 
graph}
of
$\dgg$.
Furthermore,
when $\dgg$ is an oriented graph,
we write 
$G$
to denote the underlying undirected graph of $\dgg$.
We will always denote a digraph
by a capital letter with
$\dgarw$.

\subsection{Oriented Ramsey number}
\label{sec:oriented_rms_number}
Given a graph $G$
and an oriented graph $\dgh$,
let us write
$G \rms \dgh$
to denote that
every orientation of 
the edges
of $G$
contains a copy of $\dgh$.
Since every 
graph 
admits
an acyclic orientation,
it is not possible for 
$G \rms \dgh$  to occur if $\dgh$
contains directed cycles.
In other words,
the oriented graph
$\dgh$ must be acyclic.
One may also ask for bounds on the
\emph{oriented Ramsey number}
$\dr(\dgh)$,
which is 
defined as
\begin{equation*}
    \dr(\dgh)
    :=
    \inf\set{
        n \in \bbn : 
        \text{there exists a graph } 
        G = G^n
        \text { such that }
        G \rms \dgh
    }.
\end{equation*}
We can define
$\dr(\dgh)$
equivalently
as the smallest natural number $n$ 
such that
every tournament on $n$ vertices contains $\dgh$.
To the best of our knowledge, this number was first studied by
Erd\H{o}s and Moser, who proved the following theorem.
\begin{theorem}[Erd\H{o}s and Moser~\cite{erdos_moser_tournaments}]
    Let 
    $\dgk_k$
    be the transitive tournament on $k$ vertices.
    We have
    \begin{equation*}
        2^{(k-1)/2}
        \leq
        \dr(\dgk_k)
        \leq
        2^{k-1}.
    \end{equation*}
\end{theorem}
Since every acyclic oriented graph $\dgh$
is contained in the transitive
tournament with 
the same number of
vertices,
this implies that $\dr(\dgh)$ is finite.

We now briefly survey a few bounds for the oriented Ramsey number of
orientations of paths, cycles and trees.
Let $\dgh$ be the directed path on $k$ vertices.
A well-known result obtained independently by
Gallai, Hasse, Roy and Vitaver 
(see, for example, 
Theorem 14.5 of 
Bondy and Murty~\cite{bondy_murty})
implies that
$\dr(\dgh) = k$.

Now suppose that $\dgh$ is any orientation of the undirected path with
$k \geq 2^{128}$ vertices.
The result above
was improved by Thomason~\cite{thomason_paths_cycles},
who showed that
we also have $\dr(\dgh) = k$ in this case.
This condition was later weakened to~$k \geq 9$
by Havet and Thomassé~\cite{havet_thomasse}.

Suppose now that 
$\dgh$ is an acyclic orientation of the cycle on $k$ vertices.
The same work due to
Thomason~\cite{thomason_paths_cycles}
also
showed that,
if $k \geq 2^{128}$,
then $\dr(\dgh)=k$.
Havet and Thomassé's work~\cite{havet_thomasse}
again reduced this requirement to $k \geq 68$.
Moreover,
Heydemann, Sotteau and Thomassen~\cite{orientations_hamilton_cycles}
proved that,
if a graph $G$ has at least 
$(k-1)(k-2)+3$ edges,
then $G \rms \dgh$.
One can therefore do a simple calculation to show that
$\dr(\dgh) \leq \sqrt{2}k$,
thus proving that a linear bound for the oriented Ramsey number
also holds for cycles of size smaller than $68$.

Suppose that 
$\dgh$ is an oriented tree on $k$ vertices.
A conjecture due to Sumner (see~\cite{haggkvist_thomason})
says that
$\dr(\dgh) \leq 2k-2$.
This inequality is optimal, as there are
oriented trees for which their oriented Ramsey number is at least
$2k-2$.
A linear upper bound to $\dr(\dgh)$ was first proved by
H\"{a}ggkvist and Thomason~\cite{haggkvist_thomason}, who showed that
$\dr(\dgh) \leq 12k$.
El Sahili~\cite{sahili_oriented_trees}
proved that $\dr(\dgh) \leq 3k-3$ for all oriented trees on $k$ vertices,
though its known that $\dr(\dgh)=k$ for almost all such
trees, as proved by Mycroft and Naia~\cite{naia_unavoiadble_trees}.
El Sahili's bound is currently the best upper bound when $k$ is small,
but, for large enough $k$,
Sumner's conjecture was proved by K\"{u}hn, Mycroft
and Osthus~\cite{proof_sumner_conjecture}.
A good survey about the existing work on bounds for the oriented Ramsey number
of trees can be found in the Introduction of~\cite{naia_phdthesis}.

Finally, we remark that
a work of Bloom and Burr~\cite{bloom_burr_unavoidable}
connects the oriented Ramsey number of an oriented graph with 
its family of homomorphisms.
We also remark that, in most of the literature, the oriented Ramsey number is not
studied with the language of Ramsey theory, but rather it appears under the
guise of ``unavoidable digraphs'' and related terminology.
We believe that connecting this notion to the well-studied Ramsey theory can
provide many benefits.

Unlike the classical Ramsey number,
about which much is known,
little else 
has been published on bounds 
for 
the oriented Ramsey number.
In Section~\ref{sec:oriented_ramsey},
we
apply results and concepts from
Conlon, Fox, Lee, and Sudakov~\cite{conlon_ordered}
and
Balko, 
Cibulka,
Kr{\'{a}}l 
and
Kyn\v{c}l~\cite{balko_ordered},
so as to show that
$\dr(\dgh) \leq 2R(H)^{c \log^2 h}$.
This gives an automatic upper bound on the oriented Ramsey number of
every oriented graph by considering the Ramsey number of its underlying
undirected graph, with only logarithmic loss in the exponent.
We leave open the question of whether there exists an oriented graph for
which this bound is tight.

\subsection{An oriented Ramsey theorem for random graphs}
\label{subsec:intro_random}
For a graph $H$,
we denote by $m_2(H)$ its
\defx{$2$-density},
defined as
\begin{equation*}
    m_2(H) := 
    \max_{F \sseq H, v(F) \geq 3}
    \frac{e(F)-1}{v(F)-2}.
\end{equation*}
Consider also the
binomial random graph $G(n,p)$,
which is the random graph
of order~$n$
in which each edge appears
independently
with probability $p$.
A celebrated result of
R\"odl and Ruci\'nski~\cite{rodl_threshold}
determined,
for an undirected graph $H$,
the threshold function for
$G(n,p) \rms H$.
(Here we state only the $1$-statement.)

\begin{theorem}[R\"odl and Ruci\'nski~\cite{rodl_threshold}]
    \label{thm:rodl_rucinski}
    Let $H$ be a graph.
    There exists a constant 
    $C = C(H)$ 
    such that,
    if 
    $p \geq Cn^{-1/m_2(H)}$, 
    then
    \begin{equation*}
        \lim_{n \to \infty} \bbp[G(n,p) \rms H] = 1.
    \end{equation*}
\end{theorem}

Define
$m_2(\dgh) := m_2(H)$.
In
Section~\ref{sec:oriented_random},
we prove the following version 
of Theorem~\ref{thm:rodl_rucinski}
for acyclic oriented graphs.
\begin{theorem*}
    Let $\dgh$ be an acyclic oriented graph.
    There exists a constant $C = C(\dgh)$ such that,
    if 
    $p \geq Cn^{-1/m_2(\dgh)}$, 
    then
    \begin{equation*}
        \lim_{n \to \infty} 
        \bbp\left[G(n,p) \rms \dgh\right] 
        = 1.
    \end{equation*}
\end{theorem*}

Adapting arguments from
Nenadov and Steger~\cite{nenadov_steger},
our proof of Theorem~\ref{thm:random_dir_rms} makes
use of the hypergraph container lemma of
Balogh, Morris and Samotij~\cite{hypergraphs_morris} 
and 
Saxton and Thomason~\cite{saxton_containers}.
In Section~\ref{sec:containers},
we develop the necessary container theory for digraphs
that allows us to
prove
Theorem~\ref{thm:random_dir_rms}
in Section~\ref{sec:oriented_random}.

The technique of using hypergraph containers in random
graphs for Ramsey problems 
has recently been 
employed
by
{H{\`a}n},
{Retter},
{R{\"o}dl}
and
{Schacht}~\cite{han_ramsey_containers},
{R{\"o}dl},
{Ruci{\'{n}}ski} 
and
{Schacht}~\cite{rodl_folkman}
and
{Conlon}, 
{Dellamonica}, 
{La Fleur},
{R{\"o}dl} and 
{Schacht}~\cite{induced_container}.
Our approach is also inspired by theirs,
and
some resemblance to their arguments is to be expected.

\subsection{Isometric oriented Ramsey number}
Finally, we consider the
\emph{isometric oriented Ramsey number}
$\dir(\dgh)$
of an acyclic oriented graph $\dgh$,
a concept first introduced by
{Banakh}, 
{Idzik}, 
{Pikhurko}, 
{Protasov}
and
{Pszczo{\l}a}~\cite{banakh16:_isomet}.

For an undirected graph
$G$,
we denote by
$d_G(u,v)$
the distance between
two vertices
$u, v \in V(G)$.
Given
two oriented graphs
$\dgh$
and
$\dgf$,
we say
that a 
copy
$f : V(\dgh) \to V(\dgf)$
of
$\dgh$
in
$\dgf$
is
an
\defx{isometric copy}
if
$d_H(x,y) = d_F(f(x),f(y))$
for
every
$x,y \in V(\dgh)$.
Note that the distance is taken
with respect to
the
underlying undirected graphs.

Given an oriented graph $\dgh$
and a graph $G$,
we write
$G \irms \dgh$
if
every orientation of $G$
has an isometric oriented copy of
$\dgh$.
The
\defx{isometric oriented Ramsey number}
$\dir(\dgh)$
is defined as 
\begin{equation*}
    \dir(\dgh)
    := 
    \inf
    \set{n \in \bbn : 
        \text{there exists a graph } 
        G = G^n
        \text{ such that }
        G \irms \dgh}.
\end{equation*}
It was proved in
\cite[Theorem 2.1]{banakh16:_isomet}
that
the isometric oriented Ramsey number
of acyclic oriented graphs
is always finite.
In Section~\ref{sec:isometric},
we devise a bound for 
$\dir(\dgh)$
when
$\dgh$
is
an acyclic orientation of a cycle,
adapting a construction of
{H{\`a}n},
{Retter},
{R{\"o}dl},
and
{Schacht}~\cite{han_ramsey_containers}.
In particular, we prove the following theorem.

\begin{theorem*}
    There exists 
    a 
    positive constant 
    $c$ 
    such that
    the following holds.
    Let
    $\dgh$
    be an acyclic orientation of $C_k$
    and set
    $R := \dr(\dgh)$.
    Then
    \begin{equation}
        \label{ineq:isometric_cycle}
        \dir(\dgh) 
        \leq 
        c k^{12k^3} R^{8k^2}.
    \end{equation}
\end{theorem*}
The proof also makes use of the container results we will develop in
Section~\ref{sec:containers}.

\section{Bounds for the oriented Ramsey number}
\label{sec:oriented_ramsey}

Before stating our bounds, we introduce the concept of
ordered graphs and ordered Ramsey numbers,
recently studied in
Balko, 
Cibulka,
Kr{\'{a}}l 
and
Kyn\v{c}l~\cite{balko_ordered}
and
Conlon, Fox, Lee, and Sudakov~\cite{conlon_ordered}.

An 
\defx{ordered graph}
$G$
is a pair
$G = (G', <_G)$
where $G'$ is a graph
and $<_G$ is a total ordering of the vertices of $G'$.
For convenience we write
$V(G):=V(G')$ 
and
$E(G):=E(G')$.
When a graph $G$ is equipped with a
total ordering of its
vertices, we will simply refer to $G$ as an
ordered graph without further qualifications.

An ordered graph 
$G$ 
is said to
\defx{contain} 
an ordered graph 
$H$ 
if
there exists
a function
$\phi:V(H) \to V(G)$
such that,
for every 
$x,y \in V(H)$,
we have 
$\phi(x) <_G \phi(y)$
if and only if
$x <_H y$,
and
$\{i,j\}$ is an edge of $H$ 
only if
$\{\phi(i),\phi(j)\}$
is an edge of $G$.
In this case, we call~$\phi$ a
\defx{monotone embedding}.

If the graphs $H$ and $G$ are ordered graphs,
we write
$G \orms H$
to denote that 
every two-coloring of the edges of $G$
contains an
\emph{ordered}
monochromatic
copy of $H$.
When the graph $H$ is equipped with a total ordering,
the 
\defx{ordered Ramsey number}
$\orm(H)$
    can be defined analogously,
    as follows:
\begin{equation*}
    \orm(H)
    :=
    \inf\set{
        n \in \bbn : 
        \text{there exists a graph } 
        G
        \text{ of order $n$ such that }
        G \orms H
    }.
\end{equation*}

The following is a general bound for the ordered Ramsey
number of graph, depending on the Ramsey number of its
corresponding unordered graph.
In particular, this proves that the ordered Ramsey number
of an ordered graph is always finite.

\begin{theorem}[Conlon, Fox, Lee, and Sudakov~\cite{conlon_ordered}]
    \label{thm:ord_rms}
    There exists a constant $c$ such that,
    for every ordered graph $H$ on $n$ vertices,
    we have
    \begin{equation*}
        \orm(H) \leq R(H)^{c\log^2n}.
    \end{equation*}
\end{theorem}
More precise bounds for
$R_<(H)$
for specific classes of ordered graphs
can be found in
Conlon, Fox, Lee, and Sudakov~\cite{conlon_ordered}
and
Balko, 
Cibulka,
Kr{\'{a}}l 
and
Kyn\v{c}l~\cite{balko_ordered}.

\subsection{Our bounds}

We now give a bound for
the oriented Ramsey number of $\dgh$
depending on the Ramsey number of 
$H$.
Our proof will be inspired
in the proof of
Theorem 2.1 of~\cite{banakh16:_isomet} 
but, in reality, this idea already appeared
in
Cochand and Duchet~\cite{cochand_duchet}
and
in
R{\"o}dl and Winkler~\cite{rodl_orderings}.

\begin{theorem}
    \label{thm:dir_rms}
    There exists a constant 
    $c$ such that the following holds.
    Let 
    $\dgh$
    be
    an acyclic oriented graph
    with $h$ vertices
    and $H$ its 
    underlying undirected graph.
    There exists 
    orderings
    $<_0$ and $<_1$
    of the vertices of 
    $H$
    such that,
    for
    $H_0 = (H,<_0)$
    and
    $H_1 = (H,<_1)$,
    we have
    \begin{equation*}
        \dr(\dgh) 
        \leq 
        \orm(H_0)
        +
        \orm(H_1)
        \leq
        2R(H)^{c \log^2(h)}.
    \end{equation*}
\end{theorem}
\begin{proof}
    Let $\dgf$ be the oriented graph
    formed by two disjoint copies of $\dgh$,
    in which one has reversed edges.
    More formally,
    let $\dgf$ be the oriented graph with vertex set 
    \begin{equation*}
        V(\dgf):= V(\dgh) \times \set{0,1}
    \end{equation*}
    and
    edge set
    \begin{equation*}
        E(\dgf) := 
        \set{ 
            \( (u,0),(v,0) \), \( (v,1),(u,1) \)
        : (u,v) \in E(\dgh)}.
    \end{equation*}
    Since $\dgh$ is acyclic,
    the oriented graph $\dgf$ is also acyclic.
    Therefore, there exists an ordering~$<$
    of the vertices of
    $\dgf$ such that
    $u < v$
    if 
    $(u,v) \in E(\dgf)$.
    Let $F$ 
    be the (ordered) underlying undirected graph of
    $\dgf$ equipped with the ordering $<$.
    Let $<_0$
    be an ordering of the vertices of $H$
    such that,
    for $x, y \in V(H)$,
    we have
    $x <_0 y$
    if and only if
    $(x,0) < (y,0)$.
    Define $<_1$ analogously.
    Let 
    $H_0 := (H,<_0)$
    and
    $H_1 := (H,<_1)$.
    Clearly, we have 
    $$\orm(F) \leq \orm(H_0) + \orm(H_1).$$

    Let $\ls$ be an arbitrary ordering of the vertices of
    $K_N$. We thus consider $K_N$ to be an ordered
    complete graph.
    By Theorem~\ref{thm:ord_rms},
    there exists a number $N$ such that
    $K_N \orms F$
    and
    \begin{equation*}
        N 
        =
        \orm(F)
        \leq 
        \orm(H_0) + \orm(H_1)
        \leq
        2R(H)^{c \log^2(h)}.
    \end{equation*}

    Now it suffices to prove that
    $K_N \rms \dgh$.
    Let $\dgk$ be an arbitrary orientation of $K_N$.
    Color the edges of $K_N$ in the following way:
    an edge $\set{u,v} \in E(K_N)$ with $u \ls v$
    is colored blue if $(u,v) \in E(\dgk)$ and red
    otherwise.
    By the choice of $N$,
    there exists
    an ordered monochromatic copy of $F$ in $K_N$.
    Let 
    $\phi:V(F) \to V(K_N)$
    be the monotone embedding of this copy.
    If the copy of $F$ in $K_N$ is blue,
    then
    the set of vertices
    $\set{\phi((v,0)) : v \in V(\dgh)}$
    induces a directed copy of $\dgh$ in $\dgk$
    with the color blue.
    Otherwise,
    if the copy is red,
    then
    the set of vertices
    $\set{\phi((v,1)) : v \in V(\dgh)}$
    induces a copy
    with the color red.
    In either case we have proved $K_N \rms \dgh$, as
    desired.
\end{proof}

\begin{remark}
    The proof of
    Theorem~\ref{thm:dir_rms}
    shows that
    the orderings 
    $<_0$
    and
    $<_1$
    of
    $V(\dgh)$
    can be taken
    to be
    the topological ordering of $\dgh$
    and the reverse topological ordering of $\dgh$,
    respectively.
\end{remark}

\section{A container theorem for digraphs}
\label{sec:containers}

In preparation for the results of
Section~\ref{sec:oriented_random}
and
Section~\ref{sec:isometric},
we prove
a container lemma for digraphs and some supporting lemmas
that will be useful in both sections.

\subsection{A saturation result for oriented graphs}
\label{sec:saturation}

First we need to
prove 
a saturation result.

\begin{theorem}
    \label{thm:dir_rms_quant}
    Let $\dgh$ be an acyclic oriented graph on $h$ vertices,
    and let $R := \dr(\dgh)$.
    The following holds for every $n \geq R$.
    For every $F \sseq E(K_n)$,
    if there exists an orientation $\dgf$ of $F$ such that
    $\dgf$ has at most 
    $\big(2\binom{R}{h}\big)^{-1} \cdot \binom{n}{h}$ 
    copies of $\dgh$, then
    \begin{equation*}
        \abs{E(K_n) \sm F} 
        \geq 
        (1/2R^2) \cdot n^2.
    \end{equation*}
\end{theorem}
\begin{proof}
    For convenience,
    let 
    $\eps := \big(2\binom{R}{h}\big)^{-1}$.
    Let $F \sseq E(K_n)$ 
    be such that there exists an orientation
    $\dgf$ 
    of 
    $F$ 
    with at most 
    $\eps \binom{n}{h}$ 
    copies of
    $\dgh$.
    Let $\dgk$ be an orientation of $K_n$ which agrees
    with the orientation $\dgf$ of $F$.
    Let
    \begin{equation*}
        \cals := 
        \set{S \in \binom{V(\dgk)}{R} : 
    E(\dgk[S]) \sseq \dgf}.
    \end{equation*}
        That is,
        the family 
        $\cals$
        is the collection of
        all
        $R$-element subsets $S$ of $V(\dgk)$
        such that
        every edge of $\dgk[S]$
        is contained in $\dgf$.
    By definition of $R$,
    every 
    $R$-element
    subset of
    the vertices of 
    $\dgk$
    contains 
    at least one
    copy
    of $\dgh$.
    This means that, for every $S \in \cals$,
    there exists one copy of $\dgh$
    in
    $E(\dgk[S])$.
    Moreover, 
    every copy of 
    $\dgh$ in $\dgk$
    is contained in at most
    $\binom{n-h}{R-h}$
    $R$-element subsets.
    Therefore,
    double-counting on the pairs
    $(S, \dgh')$
    where
    $S \in \cals$
    and
    $\dgh'$ is a copy of $\dgh$
    contained in $S$
    yields
    \begin{equation*}
        \abs{\cals}
        \leq
        \eps 
        \binom{n-h}{R-h}
        \binom{n}{h}
        =
        \frac{1}{2}
        \frac{\binom{n-h}{R-h}}{\binom{R}{h}}
        \binom{n}{h}
        =
        \frac{1}{2}
        \binom{n}{R}.
    \end{equation*}
    This implies that the set
    $\ovl{\cals}$ defined as
    \begin{equation*}
        \ovl{\cals}
        := 
        \binom{V(\dgk)}{R} \sm 
        \cals 
    \end{equation*}
    satisfies
    $\abs{\ovl{\cals}} \geq (1/2)\binom{n}{R}$.
    Observe that, by definition of
    $\ovl{\cals}$,
    every set
    $S \in 
    \ovl{\cals}$
    induces at least one edge
    $e \in E(\dgk) \sm \dgf$.
    Moreover, every edge
    $e \in E(\dgk) \sm \dgf$
    is contained in at most
    $\binom{n-2}{R-2}$
    $R$-element subsets.
    Now, double-counting on the pairs
    $(S,e)$
    where 
    $S \in 
    \ovl{\cals}$
    and
    $e \in E(\dgk[S])$
    we get
    \begin{equation*}
        \abs{E(\dgk) \sm \dgf}
        \geq
        \frac{\abs{\ovl{\cals}}}{\binom{n-2}{R-2}}
        \geq
        \frac{1}{2}
        \frac {\binom{n}{R}}
              {\binom{n-2}{R-2}}
        >
        \frac{1}{2R^2}
        \cdot
        n^2.
    \end{equation*}
    The desired result now follows by observing
    $\abs{E(K_n) \sm F} 
    =
    \abs{E(\dgk) \sm \dgf}$.
\end{proof}

\subsection{The general container lemma}

Let $\calh$ be a $l$-uniform hypergraph.
For a set $J \sseq V(\calh)$, we define the
\defx{degree of $J$}
by
\begin{equation*}
    d(J) := \abs{e \in E(\calh) : J \sseq e}.
\end{equation*}
For a vertex 
$v \in V(\calh)$,
we let 
$d(v) := d(\set{v})$.
For $j \in [l]$, we also define the
\defx{maximum $j$-degree} of a vertex
$v \in V(\calh)$ by
\begin{equation*}
    d^{(j)}(v) :=
    \max\set{d(J) : v \in J \in \binom{V(\calh)}{j}}.
\end{equation*}
We denote the average of $d^{(j)}(v)$ for all $v \in V(\calh)$
by
\begin{equation*}
    d_j := 
    \frac{1}{v(\calh)}
    \sum_{v \in V(\calh)}d^{(j)}(v).
\end{equation*}
Note
that
$d_1$
is the average degree of $\calh$.
Finally, for $\tau > 0$,
we define 
$\delta_j$ 
as
\begin{equation*}
    \delta_j := \frac{d_j}{d_1 \tau^{j-1}}
\end{equation*}
and
the
\defx{co-degree function}
$\delta(\calh,\tau)$
by
\begin{equation*}
    \delta(\calh,\tau) := 
    2^{\binom{l}{2}-1}
    \sum_{j=2}^l
    {2^{-\binom{j-1}{2}}}
    \delta_j.
\end{equation*}

We now state a condensed version of the Container Lemma,
as expressed in
Saxton and Thomason~\cite{saxton_containers}. 
This version
can be found
as Theorem 2.1 in~\cite{han_ramsey_containers}.

\begin{theorem}[\cite{saxton_containers}, Corollary 3.6]
    \label{thm:container}
    Let $0< \eps, \tau < 1/2$.
    Let 
    $\calh = (V,E)$ 
    be a $l$-uniform hypergraph.
    Suppose that $\tau$ satisfies
    $\delta(\calh,\tau) \leq \eps/(12l!)$.
    Then for integers
    $K = 800l(l!)^3$
    and
    $s = \floor{K\log(1/\eps)}$
    the following holds.

    For every independent set 
    $I \sseq V$ in $\calh$
    there exists an $s$-tuple
    $T = (T_1,\dots,T_s)$
    of subsets of $V$
    and a subset
    $C = C(T) \sseq V$ depending only on $S$
    such that
    \begin{enumerate}[label=(\alph*)]
        \item 
            $\bigcup_{i \in [s]} T_i \sseq I \sseq C$,
        \item 
            $e(C) \leq \eps \cdot e(\calh)$, and
        \item 
            for every $i \in [s]$ we have 
            $\abs{T_i} \leq K\tau\abs{V}$.
    \end{enumerate}
\end{theorem}

Here we prove a version of the container lemma
for
$\dgh$-free
orientations of graphs.
First, we need the following definitions.

\begin{definition}
    Let $\dgh$ be an oriented graph
    and let $n \in \bbn$.
    Denote by
    $\dgd_n$
    the digraph
    with vertex set
    $[n]$
    and
    edge set
    \begin{equation*}
        E(\dgd_n)
        := 
        ([n] \times [n])
        \sm
        \set{(v,v) : v \in [n]}.
    \end{equation*}
    We call $\dgd_n$ the
    \emph{complete digraph}.
\end{definition}

\begin{definition}[\cite{kuhn_directed}, Definition 3.5]
    Let $\dgh$ be an oriented graph
    with $l$ edges
    and let $n \in \bbn$.
    The 
    hypergraph 
    $\cald(n, \dgh) = (\calv, \cale)$ 
    is a $l$-uniform hypergraph
    with vertex set
    $\calv := E(\dgd_n)$
    and edge set
    \begin{equation*}
        \cale :=
        \set{B \in \binom{\calv}{l} :
        \text{the edges of $B$ form a digraph isomorphic to
        $\dgh$}
    }.
\end{equation*}
\end{definition}

\begin{definition}
    Let $\dgh$ be an oriented graph
    with $h$ vertices.
    In what follows,
    we denote by
    $\emb_{\dgh} := e(\cald(h, \dgh))$
    the number of copies of $\dgh$
    in $\dgd_h$.
\end{definition}

\subsection{Checking degree conditions}
To apply Theorem~\ref{thm:container}
and prove our container theorems for digraphs,
it is first necessary to prove a bound
on 
$\delta(\cald(n, \dgh), \tau)$ 
for a suitable value of
$\tau$.
This is done by the following lemma.

\begin{lemma}
    \label{lemma:deg_cont}
    Let $\dgh$ be an oriented graph
    with $h$ vertices
    and $l \geq 2$ edges.
    Let also $D_\tau \geq 1$
    and
    write
    $\tau := D_\tau n^{-1/m_2(\dgh)}$.
    We have
    \begin{equation*}
        \delta(\cald(n, \dgh), \tau)
        \leq
        2^{\binom{l}{2}}
        h^{h-2}
        D_\tau^{-1}.
    \end{equation*}
\end{lemma}
\begin{proof}
    For convenience,
    set
    $\calh :=
    \cald(n, \dgh)$.
    Let
    $J \sseq V(\calh)$.
    Define
    \begin{equation*}
        V_J
        :=
        \bigcup_{(a,b) \in J}
        \set{a,b} \sseq [n].
    \end{equation*}
    Note that
    $(V_J,J)$
    is the subdigraph of
    $\dgd_n$
    induced by the set 
    of edges
    $J$.
    For a 
    set 
    $S \sseq [n] \sm V_J$
    such that
    $\abs{S} = h-\abs{V_J}$,
    let
    $\emb_{\dgh}(J, S)$
    denote the number
    of copies 
    $\dgf$
    of
    $\dgh$
    such that
    $V(\dgf) = V_J \cup S$
    and
    $J \sseq E(\dgf)$.
    Since
    $\emb_{\dgh}(J,S)$ 
    is the same number
    for any choice of $S$ as above,
    we write
    $\emb_{\dgh}(J)$
    for
    $\emb_{\dgh}(J,S)$.

    Note that
    $d(J)$
    is the number of copies
    of
    $\dgh$
    in
    $\dgd_n$
    which contain
    the set $J$.
    Observe now that
    \begin{equation}
        \label{eq:deg_dj}
        d(J)
        =
        \binom{n-\abs{V_J}}{h-\abs{V_J}}\emb_{\dgh}(J).
    \end{equation}
    For every $j \in [l]$,
    let
    \begin{equation}
            \label{eq:def_fj}
        f(j) 
        := 
        \min_{\dgh' \sseq \dgh,\, e(\dgh')=j}
        v(H').
    \end{equation}
    It follows
    from~(\ref{eq:deg_dj})
    that
    \begin{equation*}
        d(J)
        =
        \binom{n-\abs{V_J}}{h-\abs{V_J}}\emb_{\dgh}(J)
        \leq
        \binom{n-f(j)}{h-f(j)}\emb_{\dgh}g(J).
    \end{equation*}
    Note now that,
    for every 
    $e \in V(\calh)$,
    we have
    $d^{(1)}(e) = d(e) =
    \binom{n-2}{h-2}\emb_{\dgh}g(\set{e})$.
    Therefore, the average $d_1$ of all 
    $d^{(1)}(e)$
    satisfies
    $d_1 = 
    \binom{n-2}{h-2}\emb_{\dgh}g(\set{e})$,
    for some fixed $e \in V(\calh)$.
    It follows that
    \begin{align*}
        \frac{d(J)}{d_1}
        \leq
        \frac{\binom{n-f(j)}{h-f(j)}\emb_{\dgh}g(J)}
        {\binom{n-2}{h-2}\emb_{\dgh}g(\set{e})}
        \leq
        \frac{\binom{n-f(j)}{h-f(j)}}
        {\binom{n-2}{h-2}}
        =
        \frac{(h-2)(h-3)\dots(h-f(j)+1)}
        {(n-2)(n-3)\dots(n-f(j)+1)}
        \leq
        \( \frac{h}{n} \)^{f(j)-2}.
    \end{align*}
    Therefore, we have
    $d^{(j)}(v)/d_1 \leq h^{f(j)-2}n^{2-f(j)}$.
    Since $f(j) \leq h$, this gives us
    \begin{equation*}
        \frac{d_j}{d_1}
        =
        \frac{1}{v(\calh)}
        \sum_{v \in V(\calh)}
        \frac{d^{(j)}(v)}{d_1}
        \leq
        \frac{1}{v(\calh)}
        \sum_{v \in V(\calh)}
        h^{f(j)-2}n^{2-f(j)}
        =
        h^{f(j)-2}n^{2-f(j)}
        \leq
        h^{h-2}n^{2-f(j)}.
    \end{equation*}
    We furthermore obtain 
    \begin{align}
        \label{eq:dgc1}
        \delta_j
        =
        \frac{d_j}{d_1 \tau^{j-1}}
        \leq
        h^{h-2}n^{2-f(j)}
        \tau^{1-j}
        \leq
        h^{h-2}
        n^{2-f(j)+(j-1)/m_2(\dgh)}
        D_\tau^{1-j}.
    \end{align}
    Observe now that, by definition of $m_2(\dgh)$,
    we have
    $m_2(\dgh) \geq (j-1)/(f(j)-2)$.
    From this we may derive
    ${2-f(j)+(j-1)/m_2(\dgh)} \leq 0$.
    Therefore, we can conclude from~(\ref{eq:dgc1})
    that
    \begin{equation}
        \label{eq:dgc2}
        \delta_j
        \leq
        h^{h-2}
        D_\tau^{1-j}
        \leq
        h^{h-2}
        D_\tau^{-1}.
    \end{equation}
    Now we can finally bound the co-degree function 
    $\delta(\calh, \tau)$ by observing that
    \begin{equation*}
        \delta(\calh, \tau)
        =
        2^{\binom{l}{2}-1}
        \sum_{j=2}^l
        {2^{-\binom{j-1}{2}}}
        \delta_j
        \leq
        2^{\binom{l}{2}-1}
        h^{h-2}D_\tau^{-1}
        \sum_{j=2}^l
        2^{-\binom{j-1}{2}}
        \leq
        2^{\binom{l}{2}}
        h^{h-2}D_\tau^{-1}.
    \end{equation*}
    This finishes the proof.
\end{proof}

\subsection{A container lemma for graphs with $\dgh$-free orientations}

    For convenience,
    given numbers
    $n$, $s$ and $t$,
    define
    \begin{equation*}
        \calt(n,s,t)
        :=
        \set{(T_1,\dots,T_s)
            \in E(K_n)^s
            : 
            \abs{\bigcup_{i \in [s]}T_i} \leq t 
        }.
    \end{equation*}
We are now able to state and prove
our container lemma
for graphs admitting $\dgh$-free orientations.
\begin{theorem}
    \label{thm:container_dir_rms}
    Let
    $\dgh$ 
    be an acyclic oriented graph.
    There exists 
    a real number
    $\alpha > 0$
    and positive integers 
    $n_0$, $s$ and 
    $c$
    such that 
    the following holds
    for every $n \geq n_0$.
    For every graph $G$ on 
    $n$
    vertices
    such that 
    $G \nrms \dgh$,
    there exists
    an $s$-tuple
    $T=(T_1,\dots,T_s) \in E(G)^s$
    and 
    a set
    $C = C(T) \sseq E(K_n)$
    depending only on 
    $T$
    such that
    \begin{enumerate}[label=(\alph*)]
        \item $\bigcup_{i \in [s]} T_i \sseq E(G) \sseq C$,
        \item $\abs{E(K_n)\sm C} \geq \alpha n^2$, and
        \item
            $T \in \calt(n,s,cn^{2-1/m_2(\dgh)})$.
    \end{enumerate}
\end{theorem}

\begin{proof}
    We will apply the Container Lemma~(Theorem~\ref{thm:container}).
    Let $\dgh$ be an acyclic oriented graph
    with $h$ vertices and $l$ edges and let $R := \dr(\dgh)$.
    Let 
    $\alpha = 1/(2R^2)$.
    Let $n_0 := R$
    and 
    suppose $n \geq n_0$.
    Set 
    $\calh := \cald(n, \dgh)$
    and
    $\eps := (2 \binom{R}{h}  \emb_{\dgh})^{-1}$
    and
    let
    \begin{equation*}
        D_\tau := 
        \frac{12l!2^{\binom{l}{2}}h^{h-2}}{\eps}.
    \end{equation*}
    Moreover, let $\tau := D_\tau n^{-1/m_2(\dgh)}$.
    By
    Lemma~\ref{lemma:deg_cont},
    this yields
    $\delta(\calh, \tau) \leq \eps/(12l!)$.
    Theorem~\ref{thm:container}
    now
    gives us
    numbers
    $s$ and $K$
    for 
    $\calh$, $\eps$ and $\tau$.
    Let $G$ be a graph
    on
    $n$ vertices
    such that
    $G \nrms \dgh$.
    There exists an orientation
    $\dgg$ of $G$ such that
    $\dgg$ contains no copy of
    $\dgh$.
    Therefore,
    the set
    $E(\dgg)$
    is an independent set
    of $\calh$.
    Let
    $\dgt = (\dgt_1,\dots,\dgt_s)$
    be an $s$-tuple
    of oriented edges
    and
    $\dgc = \dgc(\dgt)$
    such as
    Theorem~\ref{thm:container}
    gives for
    $E(\dgg)$.
    For $i \in [s]$,
    let $T_i$ be the 
    underlying
    set of undirected edges
    of
    $\dgt_i$.
    Define $C$ analogously for $\dgc$.
    By item (a) of
    Theorem~\ref{thm:container},
    we have
    \begin{equation*}
        \bigcup_{i \in [s]}T_i \sseq E(G) \sseq C.
    \end{equation*}
    We have thus proved item (a).

    Observe now that
    $\emb_{\dgh}g$
    counts the number of copies
    of
    $\dgh$ 
    in any subset of $h$ vertices
    of
    $\dgd_n$,
    whence
    it follows that
    \begin{equation}
        \label{eq:copies_dgh}
        e(\calh) = \binom{n}{h}\emb_{\dgh}g.
    \end{equation}
    Therefore,
    by item (b) 
    of Theorem~\ref{thm:container}
    we conclude that $\dgc$
    has at most 
    $\eps e(\calh) = (2 \binom{R}{h})^{-1} \cdot \binom{n}{h}$
    copies of $\dgh$.
    By the choice
    of
    $\alpha$,
    Theorem~\ref{thm:dir_rms_quant}
    now gives
    \begin{equation*}
        \abs{E(K_n)\sm C} \geq \alpha n^2.
    \end{equation*}
    We have thus proved item (b).
    Finally, 
    by letting
    $c := sKD_\tau$,
    we get by item (c) of Theorem~\ref{thm:container} that
    \begin{equation*}
        \abs{\bigcup_{i \in [s]}T_i} 
        \leq
        sK\tau v(\calh)
        \leq
        cn^{2-1/m_2(\dgh)}.
    \end{equation*}
    We have thus proved item (c).
    Therefore, there exists
    an $s$-tuple $T$ and a set $C = C(T)$
    with the desired requirements.
    This finishes the proof.
\end{proof}

\section{An Oriented Ramsey Theorem for Random Graphs}
\label{sec:oriented_random}

In this section,
we prove
the following theorem,
applying the results developed in
Section~\ref{sec:containers}.

\begin{theorem}
    \label{thm:random_dir_rms}
    Let $\dgh$ be an acyclic oriented graph.
    There exists a constant $C = C(\dgh)$ such that,
    if 
    $p \geq Cn^{-1/m_2(\dgh)}$, 
    then
    \begin{equation*}
        \lim_{n \to \infty} 
        \bbp\left[G(n,p) \rms \dgh\right] 
        = 1.
    \end{equation*}
\end{theorem}

\begin{proof}
    Let 
    $\alpha$,
    $s, c$
    and
    $n_0$
    be as given by
    Theorem~\ref{thm:container_dir_rms}
    for $\dgh$.
    Suppose $n \geq n_0$
    and let
    $t := cn^{2-1/m_2(\dgh)}$.
    Set moreover
    $p := Cn^{-1/m_2(\dgh)}$,
    for some constant $C$ sufficiently large
    with respect to $c$.
    We will show that 
    $\bbp[G(n,p) \nrms \dgh] = o(1)$.

    If
    a graph
    $G$
    on $n$ vertices
    satisfies
    $G \nrms \dgh$,
    then
    by
    Theorem~\ref{thm:container_dir_rms}
    there exists
    an $s$-tuple 
    $T=(T_1,\dots,T_s) \in \calt(n,s,t)$
    and 
    a set 
    $C(T) \sseq E(K_n)$
    such that
    \begin{equation}
        \label{eq:prob1}
        \bigcup_{i \in [s]} T_i \sseq E(G) \sseq C(T)
    \end{equation}
    and
    \begin{equation*}
        \abs{E(K_n)\sm C(T)}
        \geq
        \alpha n^2.
    \end{equation*}
    Let us set for convenience
    $D(T) := E(K_n) \sm C(T)$.
    Since
    $E(G) \sseq C(T)$,
    we have
    \begin{equation}
        \label{eq:prob2}
        E(G) \cap D(T) = \emptyset.
    \end{equation}

    Let 
    $\calg$
    be the family of all graphs 
    $G$
    on
    $n$ vertices
    such that
    $G \nrms \dgh$.
    For an $s$-tuple 
    $T=(T_1,\dots,T_s) \in \calt(n,s,t)$,
    let
    \begin{equation*}
        \calg_T^{'}
        :=
        \set{G = G^n: T_i \sseq E(G) \; \forall i \in [s]},
    \end{equation*}
    and let
    \begin{equation*}
        \calg_T^{''}
        :=
        \set{G = G^n: E(G) \cap D(T) = \emptyset}.
    \end{equation*}
    Observations 
    (\ref{eq:prob1})
    and
    (\ref{eq:prob2}) 
    show that
    \begin{equation*}
        \calg
        \sseq
        \bigcup_{T \in \calt(n,s,t)}\calg_T^{'} \cap
        \calg_T^{''}.
    \end{equation*}
    As the sets
    $T_i$ and $D(T)$ have empty intersection
    for every $i \in [s]$, 
    it follows that the events
    $[G(n,p) \in \calg_T^{'}]$
    and
    $[G(n,p) \in \calg_T^{''}]$
    are independent.
    We conclude
    \begin{equation*}
        \bbp[G(n,p) \in \calg]
        \leq
        \sum_{T \in \calt(n,s,t)}
        \bbp\left[ G(n,p) \in \calg_T^{'} \right]
        \cdot
        \bbp\left[ G(n,p) \in \calg_T^{''} \right].
    \end{equation*}
    Since $\abs{D(T)} \geq \alpha n^2$
    for every
    $T \in \calt(n,s,t)$,
    we have
    \begin{equation*}
        \bbp\left[ G(n,p) \in \calg_T^{''} \right]
        \leq
        (1-p)^{\alpha n^2}
        \leq
        \exp(-\alpha n^2 p).
    \end{equation*}
    Moreover, we also have
    \begin{equation*}
        \sum_{T \in \calt(n,s,t)}
        \bbp\left[ G(n,p) \in \calg_T^{'} \right]
        \leq
        \sum_{T \in \calt(n,s,t)}
        p^{\abs{\bigcup_{i \in [s]} T_i}}.
    \end{equation*}
    It follows that
    \begin{equation}
        \label{eq:sb}
        \bbp[G(n,p) \in \calg]
        \leq
        \exp(-\alpha n^2 p)
        \cdot
        \sum_{T \in \calt(n,s,t)}
        p^{\abs{\bigcup_{i \in [s]} T_i}}.
    \end{equation}
    We now proceed to bound the sum in
    (\ref{eq:sb}).
    For every integer $k$ 
    such that
    $0 \leq k \leq t$,
    define
    \begin{equation*}
        S(k) :=
        \set{T \in \calt(n,s,t) 
            : \abs{\bigcup_{i \in [s]} T_i} = k}.
    \end{equation*}
    Observe that
    $\abs{S(k)} = \binom{\binom{n}{2}}{k}(2^s)^k$.
    Indeed, there are
    $\binom{\binom{n}{2}}{k}$
    ways of choosing
    $k$ 
    edges from
    $E(K_n)$,
    and
    $(2^s)^k$
    ways of assigning these edges
    to the sets of the $s$-tuples,
    which gives
    the desired equation.
    Therefore,
    \begin{equation}
        \label{ineq:sum_of_ps}
        \sum_{T \in \calt(n,s,t)}
        p^{\abs{\bigcup_{i \in [s]} T_i}}
        =
        \sum_{k=0}^t
        \abs{S(k)}
        p^k
        \leq
        \sum_{k=0}^t
        \binom{\binom{n}{2}}{k}
        (2^s)^k
        p^k
        \leq
        1+
        \sum_{k=1}^t
        \(\frac{e 2^{s-1} n^2 p}{k}\)^k.
    \end{equation}

    Let
    $f(k)$
    be the function
    which maps 
    $k$
    to
    $(eb/k)^k$,
    where 
    $b = 2^{s-1}n^2p$.
    Observe that $f$ is unimodal and achieves its maximum 
    at $k=b$.
    Since
    $2^{s-1}n^2p \geq cn^2p/C=t$
    for $C$ sufficiently large
    with respect to
    $s$ and $c$,
    we obtain
    \begin{align*}
        1+
        \sum_{k=1}^{cn^2p/C}
        \(\frac{e 2^{s-1} n^2 p}{k}\)^k
        &\leq
        1+
        \frac{cn^2p}{C}
        \(\frac{C e 2^{s-1} n^2 p}{cn^2p}\)^{cn^2p/C}
        &&
        \text{$C$ sufficiently large}
        \\&\leq
        n^2
        \(\frac{C e 2^{s-1}}{c}\)^{cn^2p/C}
        &&
        \text{$C$ sufficiently large}
        \\&=
        n^2
        \exp\(\frac{cn^2p}{C} (\log C + 1 + (s-1)\log 2 -
        \log c) \)
        \\&=
        n^2
        \exp\(n^2p \frac{c (\log C + 1 + (s-1)\log 2 -
        \log c)}{C}  \)
        \\&\leq
        n^2
        \exp\(n^2p \frac{\alpha}{3}  \)
        &&
        \text{$C$ sufficiently large}
        \\&\leq
        \exp\(\frac{\alpha n^2p}{2}  \)
        &&
        \text{$n$ sufficiently large}.
    \end{align*}
    We may now conclude
    \begin{equation*}
        \bbp[G(n,p) \in \calg] 
        \leq
        \exp(-\alpha n^2p)
        \exp\(\frac{\alpha n^2p}{2}  \)
        = 
        \exp\(-\frac{\alpha n^2p}{2}  \)
        =
        o(1),
    \end{equation*}
    as desired.
\end{proof}

\section{The Isometric Oriented Ramsey Number}
\label{sec:isometric}

\subsection{Introduction}
\label{sec:iso_introduction}

Recently,
{Banakh}, 
{Idzik}, 
{Pikhurko}, 
{Protasov}
and
{Pszczo{\l}a}~\cite{banakh16:_isomet}
introduced the concept of
\emph{isometric oriented Ramsey number},
and they proved that
the isometric oriented Ramsey number
of any acyclic oriented graph is finite.
Moreover,
they posed 
the problem of estimating
$\dir(\dgh)$
for acyclic oriented graphs $\dgh$.
In this section, we
give an upper bound
on 
$\dir(\dgh)$
when
$\dgh$
is an acyclic orientation
of the cycle
on $k$ vertices $C_k$.
\begin{theorem}
    \label{thm:isometric_cycle}
    There exists 
    a 
    positive constant 
    $c_1$ 
    such that
    the following holds.
    Let
    $\dgh$
    be an acyclic orientation of $C_k$
    and set
    $R := \dr(\dgh)$.
    Then
    \begin{equation}
        \label{ineq:isometric_cycle}
        \dir(\dgh) 
        \leq 
        c_1 k^{12k^3} R^{8k^2}.
    \end{equation}
\end{theorem}

\begin{remark}
    As observed in the Section~\ref{sec:oriented_rms_number},
    we have $\dr(\dgh) \leq \sqrt{2}k$
    for every acyclic orientation $\dgh$ of the cycle $C_k$.
    In light of
    Theorem~\ref{thm:isometric_cycle},
    one readily sees 
    that
    there exists
    a 
    universal constant
    $c_1$ 
    such that
    \begin{equation*}
        \dir(\dgh) 
        \leq 
        k^{c_1 k^3}.
    \end{equation*}
\end{remark}

The approach 
employed in this section
to prove
Theorem~\ref{thm:isometric_cycle}
is 
very similar to
the proof
of Theorem 1.1 
in
{H{\`a}n},
{Retter},
{R{\"o}dl},
and
{Schacht}~\cite{han_ramsey_containers}.
We will prove a container theorem for graphs with $\dgh$-free
orientations, when $\dgh$ is an acyclic orientation of a cycle.
This will be a more refined version of Theorem~\ref{thm:container_dir_rms}
for this specific case. In particular, we will pay closer attention to the
numbers given by the container theorem.

\subsection{A container lemma for acyclic orientations of cycles}

We begin by observing that,
for every orientation $\dgh$ of the cycle $C_k$,
we have
\begin{equation}
    \label{eq:2_density}
    m_2(\dgh)
    =
    m_2(C_k)
    =
    \frac{k-1}{k-2}.
\end{equation}
This will justify
the choice of constants 
which
we will 
make in the rest of this
section.

We now prove the following lemma,
which is a
slightly improved version of
Lemma~\ref{lemma:deg_cont}
adjusted for orientations of cycles.
Our proof makes uses of 
arguments and results
from the proof of
Lemma~\ref{lemma:deg_cont}.

\begin{lemma}
    \label{lemma:deg_cont_cycles}
    Let $\dgh$ be 
    an orientation
    of the cycle $C_k$.
    Let also $D_\tau \geq 1$
    and
    define
    $\tau$ as
    $\tau 
    := 
    D_\tau n^{-(k-2)/(k-1)}$.
    For every
    $n \geq D_\tau ^{(k-1)^2}$,
    we have
    \begin{equation*}
        \delta(\cald(n, \dgh), \tau)
        \leq
        2^{\binom{k}{2}}
        k^{k-2}
        D_\tau^{-(k-1)}.
    \end{equation*}
\end{lemma}
\begin{proof}
    Fix
    $j \in [k]$.
    Let
    $f(j)$ 
    be as defined in~(\ref{eq:def_fj}).
    Since $\dgh$ is 
    an orientation of 
    the cycle on $k$
    vertices,
    we have 
    $f(j)=j+1$
    for every $j \in [k-1]$
    and
    $f(k)=k$.
    Furthermore,
    by~(\ref{eq:dgc1})
    we obtain
    \begin{equation*}
        \delta_j
        \leq
        k^{k-2}
        n^{2-f(j)+(j-1)(k-2)/(k-1)}
        D_\tau^{1-j}.
    \end{equation*}
    Therefore,
    for 
    $j \in [k-1]$
    we have
    \begin{align}
        \delta_j
        &
        \leq
        k^{k-2}
        n^{1-j+(j-1)(k-2)/(k-1)}
        D_\tau^{1-j}
        \nonumber
        \\&
        =
        k^{k-2}
        n^{-(j-1)/(k-1)}
        D_\tau^{1-j}
        \nonumber
        \\&
        \leq
        k^{k-2}
        k^{-1/(k-1)}
        D_\tau^{-1}
        \nonumber
        \\&
        \label{ineq:dgcyc1}
        \leq
        k^{k-2}
        n^{-1/(k-1)}.
    \end{align}
    From inequality~(\ref{eq:dgc2})
    proved in Lemma~\ref{lemma:deg_cont},
    we obtain
    that
    \begin{align}
        \label{ineq:dgcyc2}
        \delta_k
        \leq
        k^{k-2}
        D_\tau^{-(k-1)}.
    \end{align}
    Since, by assumption,
    we have
    $n \geq D_\tau^{(k-1)^2}$,
    inequalities
    (\ref{ineq:dgcyc1})
    and
    (\ref{ineq:dgcyc2})
    now give us
    \begin{equation*}
        \max_{j \in [k]}
        \delta_j
        =
        \delta_k.
    \end{equation*}
    We therefore conclude
    \begin{equation*}
        \delta(\calh, \tau)
        =
        2^{\binom{l}{2}-1}
        \sum_{j=2}^l
        {2^{-\binom{j-1}{2}}}
        \delta_j
        \leq
        2^{\binom{l}{2}-1}
        k^{k-2}D_\tau^{-(k-1)}
        \sum_{j=2}^l
        2^{-\binom{j-1}{2}}
        \leq
        2^{\binom{l}{2}}
        k^{k-2}D_\tau^{-(k-1)},
    \end{equation*}
    as desired.
\end{proof}

We are now able to state and prove
our container lemma
for graphs admitting $\dgh$-free orientations
in the specific case when $\dgh$ is an acyclic orientation of $C_k$.

\begin{theorem}
    \label{thm:container_dir_cycle}
    Let
    $\dgh$ 
    be an acyclic orientation of $C_k$.
    There exists 
    positive integers 
    $n_0$, $s$ and 
    $c$
    such that 
    the following holds
    for every $n \geq n_0$.
    For every graph $G$ on 
    $n$
    vertices
    such that 
    $G \nrms \dgh$,
    there exists
    an $s$-tuple
    $T=(T_1,\dots,T_s) \in E(G)^s$
    and 
    a set
    $C = C(T) \sseq E(K_n)$
    depending only on 
    $T$
    such that
    \begin{enumerate}[label=(\alph*)]
        \item $\bigcup_{i \in [s]} T_i \sseq E(G) \sseq C$,
        \item $\abs{E(K_n)\sm C} \geq n^2/(2R^2)$, and
        \item
            $T \in \calt(n,s,cn^{2-(k-2)/(k-1)})$.
    \end{enumerate}
    In particular, the constants can be chosen to be the following:
    \begin{alignat}{3}
        &
        n_0
        &&= 
        D^{(k-1)^2},
        \\
        &
        s 
        &&=
        \floor{K \log(2R^k)}
        &&\leq
        1600 k^{3k+2}
        R
        ,
        \\
        &
        c
        &&=
        sKD
        ,
    \end{alignat}
    where $R = \dr(\dgh)$
    and
    \begin{alignat}{3}
        &
        D
        &&= 
        4 \cdot 2^{k/2} \cdot k^2 \cdot (2R^k)^{1/(k-1)}
        &&\leq
        8 R^2 k^{k+2},
        \\
        &
        K 
        &&= 800k(k!)^3 
        &&\leq
        800k^{3k+1}.
    \end{alignat}
\end{theorem}
\begin{proof}
    The proof is very similar to Theorem~\ref{thm:container_dir_rms}.
    Let $\eps := 1/(2R^k)$
    and suppose $n \geq n_0$.
    Observe that
    \begin{equation*}
        D^{k-1}
        =
        \(
            \frac{4 \cdot 2^{k/2} \cdot
            k^2}{\eps^{1/(k-1)}}
        \)^{k-1}
        =
        \frac{4^{k-1} \cdot 2^{\binom{k}{2}} \cdot
        k^{2(k-1)}}{\eps}
        \geq
        \frac{12 \cdot 2^{\binom{k}{2}} \cdot
        k^{k-2} \cdot k!}{\eps}.
    \end{equation*}
    Moreover, let $\tau = D n^{-(k-2)/(k-1)}$.
    Lemma~\ref{lemma:deg_cont_cycles}
    now yields
    $\delta(\cald(n, \dgh), \tau) \leq \eps/(12k!)$.
    Theorem~\ref{thm:container}
    gives us
    numbers
    $s$ and $K$
    for 
    $\calh$, $\eps$ and $\tau$,
    where
    \begin{alignat}{3}
        &
        K 
        &&= 800k(k!)^3,
        \\
        &
        s 
        &&=
        \floor{K \log(1/\eps)}
        &&
        =
        \floor{K \log(2R^k)}.
    \end{alignat}

    Now item (a) can be proved in exactly the same way as in the proof of
    Theorem~\ref{thm:container_dir_rms}.
    Item (b) can proved by observing that
    \begin{equation*}
        \eps
        =
        \frac{1}{2R^k}
        \leq
        \frac{1}{2k!\binom{R}{k}}
        \leq
        \frac{1}{2\emb_{\dgh}g\binom{R}{k}}.
    \end{equation*}
    Therefore, every container $C$ will admit an orientation which has at most
    $\eps \emb_{\dgh}g \binom{n}{k} \leq 
    (2 \binom{R}{k})^{-1} \binom{n}{k}$ copies of $\dgh$.
    Item (b) then follows from Theorem~\ref{thm:dir_rms_quant}.
    Finally, item (c) can be shown just as in the proof of
    Theorem~\ref{thm:container_dir_rms}, by letting $c := sKD$.
\end{proof}

\subsection{Proof of Theorem~\ref{thm:isometric_cycle}}
We may now proceed to the proof of
Theorem~\ref{thm:isometric_cycle}.
The proof will be as follows.
We will consider the random graph
$G(n,p)$
and, imitating the proof of
Theorem~\ref{thm:random_dir_rms},
we will prove that,
with positive probability,
we have
$G(n,p) \irms \dgh$,
for
a number $n$ that satisfies
(\ref{ineq:isometric_cycle})
and a suitable choice of
$p$.
Our strategy will be to prove that
the graph 
$G(n,p)$ 
has girth at least $k$
and satisfies
$G(n,p) \rms \dgh$
for an acyclic orientation $\dgh$ of $C_k$,
which implies
$G(n,p) \irms \dgh$.

\begin{proof}[Proof of Theorem~\ref{thm:isometric_cycle}]
    Let $n_0$, $s$ ,$c$, $D$ and $K$
    be as given by Theorem~\ref{thm:container_dir_cycle}
    for $\dgh$.
    We begin by setting the 
    following
    numbers we are going to use in
    the proof:
    \begin{alignat}{3}
        &
        D_p 
        &&= K D s^2 10 R^2 \log(5R^2) 
        ,
        \\
        &
        n 
        &&= D_p^{k^2},
        \\
        &
        p 
        &&= D_p n^{-\frac{k-2}{k-1}}.
    \end{alignat}
    Observe that, for some positive constant $c_1 > 0$,
    we have
    \begin{equation*}
        D_p
        \leq
        c_1 \cdot
        k^{10k+7} R^8
        \leq
        k^{12k}R^{8},
    \end{equation*}
    which implies
    \begin{equation*}
        n \leq c_2 k^{12k^3} R^{8k^2}.
    \end{equation*}
    Moreover, we also have $n \geq n_0$.
    Therefore, we can apply Theorem~\ref{thm:container_dir_cycle}
    to graphs with $n$ vertices.

    Let us first prove the following claim.
    The proof goes just as in the proof of
    Claim~3.1 of~\cite{han_ramsey_containers}.
    \begin{claim}
        \label{claim:girth}
        We have
        $
            \bbp[\girth(G(n,p)) \geq k]
            \geq
            \exp(-kD_p^{k-1}n).
            $
    \end{claim}
    \begin{proof}
        Let
        $\calc(n,k)$
        be the set
        of all cycles $C \sseq E(K_n)$
        of length at most $k-1$.
        Let
        \begin{equation*}
            X := 
            \abs{
            \set{C \in \calc(n,k) : C \sseq E(G(n,p)) }}
        \end{equation*}
        be the random variable
        counting the number
        of cycles of length at most $k-1$
        in $G(n,p)$.
        For each
        cycle 
        $C \sseq E(K_n)$
        of length at most $k-1$,
        let
        $X_C$ be the indicator function of the event
        $E_C := \set{C \sseq E(G(n,p))}$.
        Clearly,
        $X$ is the sum of all such $C$.
        Therefore,
        \begin{equation*}
            \label{eq:cl1}
            \bbe[X]
            =
            \sum_{C \in \calc(n,k)}
            p^{\abs{C}}
            =
            \sum_{j=3}^{k-1}
            \frac{(j-1)!}{2}
            \binom{n}{j}
            p^j
            \leq
            \sum_{j=3}^{k-1}
            \frac{(pn)^j}{2j}
            \leq
            \frac{k}{6}(pn)^{k-1}
            =
            \frac{k}{6}D_p^{k-1}n.
        \end{equation*}
        Moreover,
        the set of all graphs $G$ on $n$ vertices such that
        $C \not\sseq E(G)$ is a monotone decreasing property.
        Therefore, 
        using the FKG inequality
        and applying
        the inequality
        $1-x \geq \exp(-x/(1-x))$
        for $x \in [0,1)$,
        we get
        \begin{align*}
            \label{eq:cl2}
            \bbp[\girth(G(n,p)) \geq k]
            =
            \prod_{C \in \calc(n,k)}
            (1-p^{\abs{C}})
            \geq
            \prod_{C \in \calc(n,k)}
            \exp\(-\frac{p^{\abs{C}}}{1-p^{\abs{C}}}\)
            \geq
            \exp\(-\frac{\bbe[X]}{1-p^3}\).
        \end{align*}
        One may now easily check that
        \begin{equation*}
            1-p^3 = 1-n^{-(k-2)/(k-1) + 1/k^2} > 1/6,
        \end{equation*}
        since $n > 11$, and the claim follows.
    \end{proof}
    We now prove the following claim.
    Our proof will be similar to that of
    Theorem~\ref{thm:random_dir_rms},
    with the difference 
    that the calculations will be more
    involved.
    \begin{claim}
        \label{claim:ramsey_cycles}
        We have
        \begin{equation*}
            \label{ineq:ramsey_cycles}
            \bbp[G(n,p) \rms \dgh]
            \geq
            1-\exp\(-\frac{n^2p}{4R^2}\).
        \end{equation*}
    \end{claim}
    \begin{proof}
        Applying 
        Theorem~\ref{thm:container_dir_cycle}
        instead of Theorem~\ref{thm:container_dir_rms},
        we can follow 
        the proof of Theorem~\ref{thm:random_dir_rms}
        up until inequalities~(\ref{eq:sb}) and~(\ref{ineq:sum_of_ps}).
        We then get
        \begin{equation}
            \label{ineq:pb1}
            \bbp[G(n,p) \nrms \dgh]
            \leq
            \exp\(-\frac{n^2p}{2R^2}\)
            \(
            1+
            \sum_{j=1}^{t}
            \( \frac{e2^{s-1}n^2p}{j}\)^j \),
        \end{equation}
        where 
        $t =  sKD n^{2-(k-1)/(k-2)}$.
        We now proceed to bound the sum in
        (\ref{ineq:pb1}).
        Let
        $f(k)$
        be the function
        which maps 
        $j$
        to
        $(eb/j)^j$,
        where 
        $b = 2^{s-1}n^2p$.
        Observe that $f$ is unimodal and achieves its maximum 
        at $k=b$.
        Observe moreover that
        \begin{equation*}
            2^{s-1}n^2p
            =
            2^{s-1}D_p
            n^{-\frac{k-2}{k-1}}n^2
            \geq
            sKD
            n^{-\frac{k-2}{k-1}}n^2.
        \end{equation*}
        whence it follows
        that
        \begin{align*}
            1+
            \sum_{k=1}^{t}
            \(\frac{e 2^{s-1} n^2 p}{k}\)^k
            &\leq
            1+
            t
            \(\frac{e 2^{s-1} n^2 p}{t}\)^{t}
            =
            1+
            t
            \(\frac{e 2^{s-1} D_p}{sK D}\)^{t}.
        \end{align*}
        Moreover,
        since
        \begin{equation*}
            sKD n^{-(k-2)/(k-1)}
            =
            sKD D_p^{-k^2(k-2)/(k-1)}
            \leq
            sKD D_p^{-1}
            < 1,
        \end{equation*}
        we obtain
        \begin{align*}
            1+
            t
            \(\frac{e 2^{s-1} D_p}{sK D}\)^{t}
            &\leq
            n^2
            \(\frac{e 2^{s-1} D_p}{sK D}\)^{t}
            \\&=
            n^2
            \exp\(t 
            \cdot
            \log 
            \frac{e 2^{s-1} D_p}{sK
            D}  
            \)
            \\&=
            n^2
            \exp\(n^2p 
            \cdot
            \frac{sKD}{D_p} 
            \cdot
            \log 
            \frac{e 2^{s-1} D_p}{sK
            D}  
            \)
            \\&=
            n^2
            \exp\Bigg(
            n^2p 
            \cdot
            \frac{sKD}{D_p} 
            \(
            \log(e2^{s-1})
            +
            \log
            \frac{ D_p}{sK D}  
            \)
            \Bigg).
        \end{align*}
        Observe now that
        \begin{equation}
            \label{ineq:dtdp1}
            \frac{sKD}{D_p} 
            \log(e2^{s-1})
            =
            \frac{\log(e2^{s-1})}{s10R^2\log(5R^2)} 
            \leq
            \frac{1}{10R^2\log(5R^2)} 
            \leq
            \frac{1}{10R^2}.
        \end{equation}
        Let now
        $x := D_p/(sK D)$
        and
        set
        $y := x/s$.
        Since the function
        $\log (x)/x$ 
        is decreasing for $x > e$,
        we have
        $\log (x)/x \leq \log (y)/y$.
        Note also that
        $y = 10R^2\log(5R^2) \leq (5R^2)^2$,
        since $\log x < x/2$ for $x> 0$.
        applying
        this inequality once again,
        we obtain
        $\log y \leq \log(5R^2)$.
        These observations allow us to conclude that
        \begin{equation}
            \label{ineq:dtdp2}
            \frac{\log( D_p/(sK D))}{D_p/(sK D)}
            =
            \frac{\log x}{x}
            \leq
            \frac{\log y}{y}
            \leq
            \frac{\log (5R^2)}{10R^2\log(5R^2)}
            =
            \frac{1}{10R^2}.
        \end{equation}
        Hence, by inequalities
        (\ref{ineq:dtdp1})
        and
        (\ref{ineq:dtdp2})
        we obtain
        \begin{equation*}
            n^2
            \exp\Bigg(
            n^2p 
            \cdot
            \frac{sKD}{D_p} 
            \(
            \log(e2^{s-1})
            +
            \log
            \frac{ D_p}{sK D}  
            \)
            \Bigg)
            \leq
            n^2
            \exp
            \(
            \frac{n^2p}{5R^2}
            \).
        \end{equation*}
        Observe now that
        \begin{equation*}
            \frac{n^2p}{\log n}
            =
            \frac{D_p^{2k^2}D_pD_p^{-k^2(k-2)/(k-1)}}
            {k^2\log (D_p)}
            \geq
            \frac{D_p^{2k^2}D_pD_p^{-k^2}}
            {k^2 D_p}
            =
            \frac{D_p^{k^2}}{k^2}
            \geq
            \frac{D_p}{k^2}
            \geq
            \frac{10R^2D}{k^2}
            \geq
            40R^2.
        \end{equation*}
        From this we obtain
        \begin{equation*}
            2\log n
            \leq
            \frac{n^2p}{20R^2}
            =
            \frac{n^2p}{4R^2}
            -
            \frac{n^2p}{5R^2},
        \end{equation*}
        which implies
        \begin{equation*}
            n^2
            \exp
            \(
            \frac{n^2p}{5R^2}
            \)
            \leq
            \exp
            \(
            \frac{n^2p}{4R^2}
            \).
        \end{equation*}
        All our work so far therefore implies
        \begin{equation*}
            1+
            \sum_{j=1}^{t}
            \( \frac{e2^{s-1}n^2p}{j}\)^j 
            \leq
            \exp
            \(
            \frac{n^2p}{4R^2}
            \),
        \end{equation*}
        which, in view of~(\ref{ineq:pb1}),
        yields
        \begin{equation*}
            \bbp[G(n,p) \nrms \dgh]
            \leq
            \exp
            \(
            -
            \frac{n^2p}{4R^2}
            \).
        \end{equation*}
        This finishes the proof of the claim.
    \end{proof}
    Now, in view of
    Claim~\ref{claim:girth}
    and
    Claim~\ref{claim:ramsey_cycles},
    we can deduce
    \begin{equation}
        \label{ineq:final_cycles}
        \begin{aligned}
            \bbp[\girth{G(n,p)} \geq k \text{ and } G(n,p) \rms \dgh]
            &\geq
            \bbp[\girth(G(n,p)) \geq k]
            +
            \bbp[G(n,p) \rms \dgh]
            -1
            \\&\geq
            \exp(-kD_p^{k-1}n)
            -
            \exp\(-\frac{n^2p}{4R^2}\).
        \end{aligned}
    \end{equation}
    Since we also have
    \begin{equation*}
        \frac{n^2p}{4R^2}
        =
        n
        \cdot
        \frac{D_p^{k^2+1-k^2(1-1/(k-1))}}{4R^2}
        >
        n
        \cdot
        \frac{D_p^{k+1}}{4R^2}
        > 
        kD_p^{k-1}n,
    \end{equation*}
    we may now conclude from~(\ref{ineq:final_cycles})
    that
    \begin{equation*}
        \bbp[\girth{G(n,p)} \geq k \text{ and } G(n,p) \rms \dgh]
        > 0,
    \end{equation*}
    which finishes the proof.
\end{proof}

\section{Concluding remarks}
The landscape of Ramsey theory for oriented graphs is much less explored
than for undirected graphs.
It would be interesting
to obtain better bounds for the oriented Ramsey number 
of specific families of graphs.
In particular, we leave unaddressed the question of whether the upper
bound of Theorem~\ref{thm:dir_rms}
is tight for some oriented graph.

Moreover, one could also consider not only
orientations of graphs, but also orientations
\emph{and} colorings of edges, and require the oriented copy
to be monochromatic.  
The resulting Ramsey number was studied, for instance,
in a work due to Buci\'{c}, Letzter and
Sudakov~\cite{bucic_letzter_sudakov_oriented_colors}.
We believe that
an upper bound 
to the threshold probability of this property,
like that of Theorem~\ref{thm:random_dir_rms},
can also be proved with the techniques of this work.

Our work also leaves open the question 
of proving a lower bound for the threshold probability
of the property $\{ G(n,p) \rms \dgh \}$,
matching the upper bound of Theorem~\ref{thm:random_dir_rms}.
One can easily check that this upper bound is not tight for the
transitive tournament on~$3$ vertices
(if $G$ contains a $K_4$, then every orientation of $G$ contains the
tournament on~$3$ vertices).
It would be very interesting
to understand for which oriented graphs this upper bound is tight,
and to prove tighter bounds when this is not the case.

Finally, one could also try to apply the techniques of
Section~\ref{sec:isometric} to derive bounds for the
isometric Ramsey number of other graphs, like directed paths.

\section{Acknowledgements}
The author thanks Marcelo Tadeu de Sá Oliveira Sales 
for helpful comments in an early
stage of this work, and Yoshiharu Kohayakawa and Tássio Naia 
for proofreading an earlier version of this work
and for providing
numerous suggestions, helpful comments and encouragements.

\bibliographystyle{plain}
\bibliography{bibliography}

\begin{thebibliography}{10}

\bibitem{balko_ordered}
Martin Balko, Josef Cibulka, Karel Kr{\'{a}}l, and Jan Kyn\v{c}l.
\newblock Ramsey numbers of ordered graphs.
\newblock {\em Electronic Notes in Discrete Mathematics}, 49:419 -- 424, 2015.
\newblock The Eight European Conference on Combinatorics, Graph Theory and
  Applications, EuroComb 2015.

\bibitem{hypergraphs_morris}
J\'ozsef Balogh, Robert Morris, and Wojciech Samotij.
\newblock Independent sets in hypergraphs.
\newblock {\em J. Amer. Math. Soc.}, 28(3):669--709, 2015.

\bibitem{banakh16:_isomet}
Taras Banakh, Adam Idzik, Oleg Pikhurko, Igor Protasov, and Krzysztof
  Pszczoła.
\newblock Isometric copies of directed trees in orientations of graphs.
\newblock {\em Journal of Graph Theory}, pages 1--17, 2019.

\bibitem{bloom_burr_unavoidable}
Gary~S. Bloom and Stefan~A. Burr.
\newblock On unavoidable digraphs in orientations of graphs.
\newblock {\em J. Graph Theory}, 11(4):453--462, 1987.

\bibitem{bondy_murty}
J.~A. Bondy and U.~S.~R. Murty.
\newblock {\em Graph theory}, volume 244 of {\em Graduate Texts in
  Mathematics}.
\newblock Springer, New York, 2008.

\bibitem{bucic_letzter_sudakov_oriented_colors}
Matija Buci\'{c}, Shoham Letzter, and Benny Sudakov.
\newblock Directed {R}amsey number for trees.
\newblock {\em J. Combin. Theory Ser. B}, 137:145--177, 2019.

\bibitem{cochand_duchet}
M.~Cochand and P.~Duchet.
\newblock A few remarks on orientation of graphs and {R}amsey theory.
\newblock In {\em Irregularities of partitions ({F}ert\H od, 1986)}, volume~8
  of {\em Algorithms Combin. Study Res. Texts}, pages 39--46. Springer, Berlin,
  1989.

\bibitem{induced_container}
D.~{Conlon}, D.~{Dellamonica}, Jr., S.~{La Fleur}, V.~{R{\"o}dl}, and
  M.~{Schacht}.
\newblock {A note on induced Ramsey numbers}.
\newblock {\em arXiv e-prints}, June 2016.

\bibitem{conlon_ordered}
David Conlon, Jacob Fox, Choongbum Lee, and Benny Sudakov.
\newblock Ordered {R}amsey numbers.
\newblock {\em J. Combin. Theory Ser. B}, 122:353--383, 2017.

\bibitem{sahili_oriented_trees}
A.~El~Sahili.
\newblock Trees in tournaments.
\newblock {\em J. Combin. Theory Ser. B}, 92(1):183--187, 2004.

\bibitem{erdos_moser_tournaments}
P.~Erd{\H{o}}s and L.~Moser.
\newblock On the representation of directed graphs as unions of orderings.
\newblock {\em Magyar Tud. Akad. Mat. Kutat\'o Int. K\"ozl.}, 9:125--132, 1964.

\bibitem{ramsey_erdos2}
P.~Erd\"os and G.~Szekeres.
\newblock A combinatorial problem in geometry.
\newblock {\em Compositio Math.}, 2:463--470, 1935.

\bibitem{haggkvist_thomason}
Roland H\"{a}ggkvist and Andrew Thomason.
\newblock Trees in tournaments.
\newblock {\em Combinatorica}, 11(2):123--130, 1991.

\bibitem{han_ramsey_containers}
H.~H{\`a}n, T.~Retter, V.~R{\"o}dl, and M.~Schacht.
\newblock Ramsey-type numbers involving graphs and hypergraphs with large
  girth.
\newblock {\em Combin. Probab. Comput.}
\newblock To appear.

\bibitem{havet_thomasse}
Fr\'{e}d\'{e}ric Havet and St\'{e}phan Thomass\'{e}.
\newblock Oriented {H}amiltonian paths in tournaments: a proof of {R}osenfeld's
  conjecture.
\newblock {\em J. Combin. Theory Ser. B}, 78(2):243--273, 2000.

\bibitem{orientations_hamilton_cycles}
M.-C. Heydemann, D.~Sotteau, and C.~Thomassen.
\newblock Orientations of {H}amiltonian cycles in digraphs.
\newblock {\em Ars Combin.}, 14:3--8, 1982.

\bibitem{proof_sumner_conjecture}
Daniela K\"{u}hn, Richard Mycroft, and Deryk Osthus.
\newblock A proof of {S}umner's universal tournament conjecture for large
  tournaments.
\newblock {\em Proc. Lond. Math. Soc. (3)}, 102(4):731--766, 2011.

\bibitem{kuhn_directed}
Daniela K\"uhn, Deryk Osthus, Timothy Townsend, and Yi~Zhao.
\newblock On the structure of oriented graphs and digraphs with forbidden
  tournaments or cycles.
\newblock {\em J. Combin. Theory Ser. B}, 124:88--127, 2017.

\bibitem{naia_unavoiadble_trees}
Richard Mycroft and T\'{a}ssio Naia.
\newblock Unavoidable trees in tournaments.
\newblock {\em Random Structures Algorithms}, 53(2):352--385, 2018.

\bibitem{naia_phdthesis}
Tássio Naia.
\newblock {\em Large Structures in Dense Directed Graphs}.
\newblock PhD thesis, University of Birmingham, 7 2018.

\bibitem{nenadov_steger}
Rajko Nenadov and Angelika Steger.
\newblock A short proof of the random {R}amsey theorem.
\newblock {\em Combin. Probab. Comput.}, 25(1):130--144, 2016.

\bibitem{radz_survey_ramsey}
Stanis\l aw~P. Radziszowski.
\newblock Small {R}amsey numbers.
\newblock {\em Electron. J. Combin.}, 1:Dynamic Survey 1, 30, 1994.

\bibitem{ramsey_original}
F.~P. Ramsey.
\newblock On a problem of formal logic.
\newblock {\em Proc. London Math. Soc. (2)}, 30(4):264--286, 1929.

\bibitem{rodl_threshold}
Vojt{\v e}ch R\"odl and Andrzej Ruci\'nski.
\newblock Threshold functions for {R}amsey properties.
\newblock {\em J. Amer. Math. Soc.}, 8(4):917--942, 1995.

\bibitem{rodl_folkman}
Vojt{\v{e}}ch R{\"o}dl, Andrzej Ruci{\'{n}}ski, and Mathias Schacht.
\newblock An exponential-type upper bound for folkman numbers.
\newblock {\em Combinatorica}, May 2016.

\bibitem{rodl_orderings}
Vojt{\v e}ch R{\"o}dl and Peter Winkler.
\newblock A {R}amsey-type theorem for orderings of a graph.
\newblock {\em SIAM J. Discrete Math.}, 2(3):402--406, 1989.

\bibitem{saxton_containers}
David Saxton and Andrew Thomason.
\newblock Hypergraph containers.
\newblock {\em Invent. Math.}, 201(3):925--992, 2015.

\bibitem{thomason_paths_cycles}
Andrew Thomason.
\newblock Paths and cycles in tournaments.
\newblock {\em Trans. Amer. Math. Soc.}, 296(1):167--180, 1986.

\end{thebibliography}

\end{document}